\documentclass[12pt]{amsart}
\usepackage{amsmath}
\usepackage{amssymb, amsfonts, amsthm}
\usepackage{amssymb}
\usepackage{ifthen}
\usepackage{graphicx}
\usepackage{float}
\usepackage{caption}
\usepackage{subcaption}
\usepackage{cite}
\usepackage{amsfonts}
\usepackage{amscd}
\usepackage{amsxtra}
\usepackage{color}
\numberwithin{equation}{section}

\newtheorem{theorem}{Theorem}[section]
\newtheorem{lemma}[theorem]{Lemma}

\theoremstyle{definition}

\begin{document}
	\title[A generalization of the Bohr inequality]
	{A generalization of the Bohr inequality for bounded analytic functions on simply connected domains and its applications}
	
	\author[Shankey Kumar]{Shankey Kumar$^*$}
	\address{Shankey Kumar, Department of Mathematics,
		Indian Institute of Technology Indore,
		Indore 453552, India}
	\email{shankeygarg93@gmail.com}
	
	
	
	\keywords{Bounded analytic functions, Bohr radius, Integral operators, Simply connected domains.\\
		$^*$ The corresponding author}

	\begin{abstract}
		Bohr’s classical theorem and its generalizations are now
		active areas of research and have been the source of investigations in numerous function spaces.
		In this article, we study a generalized Bohr’s inequality for the class of bounded analytic functions defined on the simply connected domain
		$$
		\Omega_{\gamma}:=\bigg\{z\in \mathbb{C}: \bigg|z+\frac{\gamma}{1-\gamma}\bigg|<\frac{1}{1-\gamma}\bigg\}, \,\ \text{for } 0\leq \gamma<1.
		$$
		Part of its applications, we calculate the Bohr-type radii for some known integral operators.
	\end{abstract}
	
	\maketitle
	
	\section{\bf Introduction}
	Let $\mathbb{D}:=\{z\in \mathbb{C}:|z|<1 \}$ be the unit disk in the complex plane $\mathbb{C}$ and $\Omega$ denote a simply connected domain containing $\mathbb{D}$.
	Let $\mathcal{H}(\Omega)$ be the class of all analytic functions defined on $\Omega$. The subclass 
	$\mathcal{B}(\Omega)=\{f\in\mathcal{H}(\Omega): |f(z)|\leq 1\}$,
	of $\mathcal{H}(\Omega)$, is our main consideration in this paper.
	
	In 2010, Fournier and Ruscheweyh \cite{FR10} introduced the Bohr radius for the family $\mathcal{B}(\Omega)$ by computing the positive real number
	$B_\Omega\in(0,1)$ defined as 
	$$
	B_\Omega = \sup\Big\{r \in(0, 1) : \sum_{n=0}^{\infty} |a_n| r^n \leq 1 \text{ for all } f(z)=\sum_{n=0}^{\infty} a_n z^n\in \mathcal{B}(\Omega),\ |z|<1\Big\}.
	$$
	The choice $\Omega= \mathbb{D}$, gives the well-known result of Bohr \cite{Bohr14} in which he provides that $B_\mathbb{D} = 1/3$. Bohr's result in its final form says the following: if $f(z)=\sum_{n=0}^{\infty} a_n z^n\in \mathcal{B}(\mathbb{D})$ then
	\begin{equation}\label{6eq0.1}
		\sum_{n=0}^{\infty} |a_n| r^n\leq 1
	\end{equation}
	in $|z|=r\leq 1/3$. The constant $1/3$ is known as the {\em Bohr radius}.  Bohr's paper indicates that Bohr initially obtained the radius $1/6$, but this was quickly improved to the sharp result by Riesz, Schur, and Wiener, independently. For background information about this inequality and further work related to Bohr's radius, we refer survey article \cite{Abu-M16} and references therein. Moreover, to find certain recent results, we refer to \cite{Abu4,AAL20,BD19,Kaypon18,KS21,LLP2020,LPW2020,LP2019}. 
	
	Many interesting extensions of Bohr's inequality in various settings have been developed by several mathematicians. Some interesting extension of the classical Bohr's radius problem were given by Boas and Khavinson \cite{BK97} in several complex variables. Further, Aizenberg \cite{A00,A07,A12}, Aizenberg et al.\cite{AES05}, and Liu and Ponnusamy \cite{LP21} carried out this work in different settings. Also, the notion of the Bohr radius was generalized in \cite{Abu,Abu2,A07} to include mappings from $\mathbb{D}$ to some other domains in $\mathbb{C}$. Moreover, the Bohr phenomenon for shifted disks and simply connected domains are discussed in \cite{AAH21,EPR21,FR10}. Bohr type inequalities for certain integral operators have been obtained in \cite{Kayponn19,KS2021}.
	
	In \cite{KSS2017}, Kayumov et al. presented the Bohr radius for locally univalent planar harmonic mappings. Part of the recent development in this direction, several improved versions of the Bohr inequality for harmonic mappings are discussed by Evdoridis et al. in \cite{EPR19}. Various improved forms of the classical Bohr inequality were investigated by Kayumov and Ponnusamy in \cite{Kayumov18,Kayponn18}. Recently, a generalized form of the Bohr sum is studied by Kayumov et. al. \cite{Kayponn20}, which is described as follows:
	let $\{\phi_k(r)\}_{k=0}^{\infty}$ be a sequence of non-negative continuous functions in $[0,1)$ such that the series 
	$$
	\phi_0(r)+\sum_{k=1}^{\infty}\phi_k(r)
	$$
	converges locally uniformly for $r\in[0,1)$.
	
	\noindent
	{\bf Theorem A \cite{Kayponn20}.}
	{\em Let $f(z)=\sum_{n=0}^{\infty} a_n z^n \in \mathcal{B}(\mathbb{D})$ and $p\in(0,2]$. If
		$$
		\phi_0(r)>\frac{2}{p}\sum_{k=1}^{\infty}\phi_k(r)
		\quad \mbox{ for $r\in [0,R)$},
		$$
		where $R$ is the minimal positive root of the equation 
		$$
		\phi_0(x)=\frac{2}{p}\sum_{k=1}^{\infty}\phi_k(x),
		$$
		then the following sharp inequality holds:
		$$
		|a_0|^p \phi_0(r)+\sum_{k=1}^{\infty}|a_k| \phi_k(r)\leq \phi_0(r), \quad \mbox{ for all $r\leq R$}.
		$$
		In the case when 
		$$
		\phi_0(x)<\frac{2}{p}\sum_{k=1}^{\infty}\phi_k(x)
		$$
		in some interval $(R,R+\epsilon)$, the number $R$ cannot be improved. If the functions $\phi_k(x)$ ($k\geq 0$) are smooth functions then the last condition is equivalent to the inequality 
		$$
		\phi_0^{'}(R)<\frac{2}{p}\sum_{k=1}^{\infty}\phi_k^{'}(R).
		$$}
	
	In this note our aim is to generalize Theorem A for functions in the class $\mathcal{B}(\Omega_{\gamma})$, where
	$$
	\Omega_{\gamma}:=\bigg\{z\in \mathbb{C}: \bigg|z+\frac{\gamma}{1-\gamma}\bigg|<\frac{1}{1-\gamma}\bigg\}, \,\ \text{for } 0\leq \gamma<1.
	$$
	Note that $\mathbb{D}=\Omega_{0}$. The Bohr inequality for functions in the class $\mathcal{B}(\Omega_{\gamma})$ were extended by Fournier and Ruscheweyh \cite{FR10}.
	
	We arrange this paper as follows. In Section 2, we state and prove our main result that describes Theorem A for the class $\mathcal{B}(\Omega_{\gamma})$. Finally, in Section 3 with the help of our main result, we investigate several versions of Bohr's inequality for functions in $\mathcal{B}(\Omega_{\gamma})$, and the sharp Bohr type inequality for $\beta$-Ces\'{a}ro operator, $\alpha$-Ces\'{a}ro operator  and Bernardi integral operator for functions in $\mathcal{B}(\Omega_{\gamma})$.  
	
	\section{\bf Main result}
	We start this section with a lemma obtained by Evdoridis et al. \cite{EPR21} which plays a crucial role to prove our main result.
	\begin{lemma}\label{8lemma1.1}
		For $\gamma \in [0,1)$, if $f\in\mathcal{B}(\Omega_\gamma)$ has the series representation $f(z)=\sum_{n=0}^{\infty} a_n z^n$ in $\mathbb{D}$  then
		$$
		|a_n|\leq \frac{1-|a_0|^2}{1+\gamma}, \,\ \text{ for } n\geq 1.
		$$
	\end{lemma}
	Now, we are able to prove our main result.
	\begin{theorem}\label{8theorem1.1}
		Let $\{\varphi_k(r)\}_{k=0}^{\infty}$ be a sequence of non-negative continuous functions in $[0,1)$ such that the series 
		$$
		\varphi_0(r)+\sum_{k=1}^{\infty}\varphi_k(r)
		$$
		converges locally uniformly with respect to $r\in[0,1)$. 
		For $\gamma \in [0,1)$, let $f\in\mathcal{B}(\Omega_\gamma)$ have the series representation $f(z)=\sum_{n=0}^{\infty} a_n z^n$ in $\mathbb{D}$ and $p\in(0,2]$. If
		\begin{equation}\label{6eq1.1}
			(1+\gamma)\varphi_0(r)>\frac{2}{p}\sum_{k=1}^{\infty}\varphi_k(r)
		\end{equation}
		then the following sharp inequality holds:
		$$
		A_f(\varphi,p,r):=|a_0|^p \varphi_0(r)+\sum_{k=1}^{\infty}|a_k| \varphi_k(r)\leq \varphi_0(r), \mbox{ for all $|z|=r\leq R_\gamma$},
		$$
		where $R_\gamma$ is the minimal positive root of the equation 
		$$
		(1+\gamma)\varphi_0(x)=\frac{2}{p}\sum_{k=1}^{\infty}\varphi_k(x).
		$$
		In the case when 
		$$
		(1+\gamma)\varphi_0(x)<\frac{2}{p}\sum_{k=1}^{\infty}\varphi_k(x)
		$$
		in some interval $(R_\gamma,R_\gamma+\epsilon)$, the number $R_\gamma$ cannot be improved. If the functions $\varphi_k(x)$ ($k\geq 0$) are smooth then the last condition is equivalent to the inequality 
		$$
		(1+\gamma)\varphi_0^{'}(x)<\frac{2}{p}\sum_{k=1}^{\infty}\varphi_k^{'}(x).
		$$
	\end{theorem}
	\begin{proof} Let $|a_0|<1$.
		Given that $f\in\mathcal{B}(\Omega_\gamma)$. Then by using Lemma \ref{8lemma1.1} we obtain
		\begin{align*}
			A_f(\varphi,p,r)&\leq |a_0|^p \varphi_0(r)+\frac{1-|a_0|^2}{1+\gamma}\sum_{k=1}^{\infty} \varphi_k(r)\\
			&= \varphi_0(r)+\frac{1-|a_0|^2}{1+\gamma}\Bigg[\sum_{k=1}^{\infty} \varphi_k(r)-(1+\gamma)\frac{1-|a_0|^p}{1-|a_0|^2}\varphi_0(r)\Bigg].
		\end{align*}	
		To proceed further in the proof, we use the following inequality proved in \cite{Kayponn20}
		$$
		\frac{1-x^p}{1-x^2}\geq\frac{p}{2}, \ \text{ for all } x \in [0,1)  \text{ and } p\in(0,2].
		$$
		Then we have
		$$
		A_f(\varphi,p,r)\leq \varphi_0(r)+\frac{1-|a_0|^2}{1+\gamma}\Bigg[\sum_{k=1}^{\infty} \varphi_k(r)-(\gamma+1)\frac{p}{2}\varphi_0(r)\Bigg].
		$$
		The equation \eqref{6eq1.1} provides us
		$$
		A_f(\varphi,p,r)\leq \varphi_0(r),\quad \mbox{for all $ r\leq R_\gamma$}.
		$$
		Now, let us prove the sharpness part. For $a\in [0,1)$, we consider the function
		$$
		g_\gamma(z)=\frac{a-\gamma-(1-\gamma)z}{1-a\gamma-a(1-\gamma)z}=\frac{a-\gamma}{1-a\gamma}-\sum_{k=1}^{\infty}\frac{1-a^2}{a(1-a\gamma)} \bigg(\frac{a(1-\gamma)}{1-a\gamma}\bigg)^k z^{k}, \ z\in\mathbb{D}
		$$
		which maps $\Omega_\gamma$ univalently onto $\mathbb{D}$.
		Let $\gamma$ be such that $0\leq\gamma<a$. Then we obtain
		\begin{align*}
			A_{g_\gamma}(\varphi,p,r)=& \bigg(\frac{a-\gamma}{1-a\gamma} \bigg)^p \varphi_0(r)+\sum_{k=1}^{\infty}\frac{1-a^2}{a(1-a\gamma)} \bigg(\frac{a(1-\gamma)}{1-a\gamma}\bigg)^k \varphi_k(r)\\
			=& \varphi_0(r)+\frac{(1-a)}{(1-\gamma)}\Bigg[2\sum_{k=1}^{\infty} \varphi_k(r)-p(1+\gamma)\varphi_0(r)\Bigg]\\
			&+(1-a)\Bigg[\sum_{k=1}^{\infty}\frac{1+a}{a(1-a\gamma)} \bigg(\frac{a(1-\gamma)}{1-a\gamma}\bigg)^k \varphi_k(r)-\frac{2}{(1-\gamma)}\sum_{k=1}^{\infty} \varphi_k(r)\Bigg]\\
			&+\Bigg[p(1-a)\frac{1+\gamma}{1-\gamma}+\bigg(\frac{a-\gamma}{1-a\gamma} \bigg)^p-1\Bigg]\varphi_0(r)\\
			=&\varphi_0(r)+\frac{(1-a)}{(1-\gamma)}\Bigg[2\sum_{k=1}^{\infty} \varphi_k(r)-p(1+\gamma)\varphi_0(r)\Bigg]+O((1-a)^2)
		\end{align*}
		as $a$ tends to $1^-$. Since, we have
		$$
		2\sum_{k=1}^{\infty} \varphi_k(r)>p(1+\gamma)\varphi_0(r)
		$$
		for $r\in(R_\gamma,R_\gamma+\epsilon)$, the radius can not be improved. This completes the proof.
	\end{proof}
	\section{\bf Applications}
	As applications of Theorem \ref{8theorem1.1}, the following results are the counterparts of Bohr's theorem in various settings:
	\subsection{}
	Let $\varphi_0=1$, $\varphi_n=r^n$, $n\geq N\in\mathbb{N}$ and $\varphi_n=0$, $1\leq n< N$ in Theorem \ref{8theorem1.1}. Then we obtain
	$$
	|a_0|^p+\sum_{n=N}^{\infty} |a_n| r^n\leq 1, \text{ for } r\leq R_1^{\gamma,N}(p),
	$$
	where $R_1^{\gamma,N}(p)$ is the smallest positive root of the equation $2x^N-p(1+\gamma)(1-x)=0$. The radius $R_1^{\gamma,N}(p)$ is best possible. The case $N=1$ and $p=1$ produce  \cite[Theorem 1]{FR10}.
	
	\subsection{}
	The choices $\varphi_{2n}=r^{2n}$ and $\varphi_{2n+1}=0$ ($n\geq 0$) in Theorem $\ref{8theorem1.1}$ give
	$$
	|a_0|^p+\sum_{n=1}^{\infty} |a_{2n}| r^{2n}\leq 1, \text{ for } r\leq R_2^\gamma(p),
	$$
	where $R_2^\gamma(p)$ is the positive root of the equation $2x^2-p(1+\gamma)(1-x^2)=0$. The radius $R_2^\gamma(p)$ is best possible.
	
	\subsection{}
	Let us consider $\varphi_0=1$, $\varphi_{2n}=0$ and $\varphi_{2n-1}=r^{2n-1}$ ($n\geq 1$) in Theorem \ref{8theorem1.1}. Then we have
	$$
	|a_0|^p+\sum_{n=1}^{\infty} |a_{2n-1}| r^{2n-1}\leq 1, \text{ for } r\leq R_3^\gamma(p),
	$$
	where $R_3^\gamma(p)$ is the minimal positive root of the equation $2x-p(1+\gamma)(1-x^2)=0$. The radius $R_3^\gamma(p)$ is best possible. 
	
	\subsection{}
	After considering $\varphi_0=1$, $\varphi_n=(n+1)r^n$, $n\geq N\in\mathbb{N}$ and $\varphi_n=0$, $1\leq n< N$, in Theorem \ref{8theorem1.1}, we obtain
	$$
	|a_0|^p+\sum_{n=N}^{\infty}(n+1) |a_n| r^n\leq 1, \text{ for } r\leq R_4^{\gamma,N}(p),
	$$
	where $R_4^{\gamma,N}(p)$ is the smallest positive root of the equation $2x^N(1+N-Nx)-p(1+\gamma)(1-x)^2=0$. The radius $R_4^{\gamma,N}(p)$ is best possible.
	
	\subsection{}
	It is not difficult to calculate that
	$$
	\sum_{n=N}^{\infty}n r^n=\frac{r^N[N(1-r)+r]}{(1-r)^2}.
	$$
	Then, for $\varphi_0=1$, $\varphi_n=nr^n$, $n\geq N\in\mathbb{N}$ and $\varphi_n=0$, $1\leq n< N$, Theorem \ref{8theorem1.1} obtains
	$$
	|a_0|^p+\sum_{n=N}^{\infty}n |a_n| r^n\leq 1, \text{ for } r\leq R_5^{\gamma,N}(p),
	$$
	where $R_5^{\gamma,N}(p)$ is the positive root of the equation $2x^N[N(1-x)+x]=p(1+\gamma)(1-x)^2$. The radius $R_5^{\gamma,N}(p)$ is best possible.
	
	Also, we have
	$$
	\sum_{n=N}^{\infty}n^2 r^n=\frac{r^N[(r+N)^2+r+N^2r^2-2Nr(r+N)]}{(1-r)^3}.
	$$
	Then putting $\varphi_0=1$, $\varphi_n=n^2r^n$, $n\geq N\in\mathbb{N}$ and $\varphi_n=0$, $1\leq n< N$, in Theorem \ref{8theorem1.1}, one obtains
	$$
	|a_0|^p+\sum_{n=N}^{\infty}n^2 |a_n| r^n\leq 1, \text{ for } r\leq R_6^{\gamma,N}(p),
	$$
	where $R_6^{\gamma,N}(p)$ is the positive root of the equation 
	$$
	2x^N[(x+N)^2+x+N^2x^2-2Nx(x+N)]=p(1+\gamma)(1-x)^3.
	$$ 
	The radius $R_6^{\gamma,N}(p)$ is best possible.

	\subsection{The $\beta$-Ces\'{a}ro operator} The {\em $\beta$-Ces\'aro operator} ($\beta>0$) \cite{KS2020,KS2021} is defined by
	$$
	T_\beta[f](z):=\sum_{n=0}^{\infty} \bigg(\frac{1}{n+1}\sum_{k=0}^{n} \frac{\Gamma{(n-k+\beta)}}
	{\Gamma{(n-k+1)}\Gamma(\beta)}a_{k}\bigg) z^n
	=\int_{0}^{1} \frac{f(tz)}{(1-tz)^\beta} dt, \ z\in \mathbb{D}.
	$$
	The $\beta$-Ces\'aro operator $T_\beta$ ($\beta>0$) is a natural 
	generalization of the classical Ces\'{a}ro operator $T_1$.
	For $f\in \mathcal{B}(\mathbb{D})$ and $\beta>0$, an easy computation of the integral in absolute value 
	gives us the sharp inequality
	$$ 
	|T_\beta[g](z)|\leq 
	\begin{cases}\cfrac{1}{r}\bigg[ \cfrac{1-(1-r)^{1-\beta}}{1-\beta}\bigg], & \text{if } \beta\neq 1,\\[5mm] 
		\cfrac{1}{r}\log \cfrac{1}{1-r}, & \text{ if } \beta =1, 
	\end{cases}
	$$
	for each $|z|=r<1$. 
	
	As an application of Theorem \ref{8theorem1.1} we obtain the Bohr type inequality for the $\beta$-Ces\'{a}ro operator in the following theorem. 
	
	\begin{theorem}\label{8theorem3.1}
		For $f(z)=\sum_{n=0}^{\infty}a_n z^n \in \mathcal{B}(\Omega_\gamma)$ and $0<\beta \neq 1$, we have
		$$
		\sum_{n=0}^{\infty} \bigg(\frac{1}{n+1}\sum_{k=0}^{n} \frac{\Gamma{(n-k+\beta)}}
		{\Gamma{(n-k+1)}\Gamma(\beta)}|a_{k}|\bigg) r^n
		\leq \frac{1}{r}\bigg[\frac{1-(1-r)^{1-\beta}}{1-\beta}\bigg],
		$$
		for $r\leq R_\gamma(\beta)$, where $R_\gamma(\beta)$ is the smallest positive root of the equation
		$$
		\frac{(3+\gamma)[1-(1-x)^{1-\beta}]}{1-\beta} - \frac{2[(1-x)^{-\beta}-1]}{\beta} =0.
		$$
		The radius $R_\gamma(\beta)$ cannot be improved.
	\end{theorem}
	\begin{proof}
		It is easy to observe that 
		$$
		\sum_{n=0}^{\infty} \bigg(\frac{1}{n+1}\sum_{k=0}^{n} \frac{\Gamma{(n-k+\beta)}}
		{\Gamma{(n-k+1)}\Gamma(\beta)}|a_{k}|\bigg) r^n = \sum_{n=0}^{\infty}|a_{n}| \bigg(\sum_{k=n}^{\infty} \frac{1}{k+1} \frac{\Gamma{(k-n+\beta)}}
		{\Gamma{(k-n+1)}\Gamma(\beta)}r^k\bigg). 
		$$
		Let
		$$
		\varphi_n(r)= \sum_{k=n}^{\infty} \frac{1}{k+1} \frac{\Gamma{(k-n+\beta)}}
		{\Gamma{(k-n+1)}\Gamma(\beta)}r^k.
		$$
		If we consider $f(z)=1/(1-z)$ in the definition of the operator $T_\beta$, then it is easy to find that 
		$$
		\sum_{n=0}^{\infty}\varphi_n(r)=\frac{1}{\beta r}[(1-r)^{-\beta}-1].
		$$
		Also, we get that  
		$$
		\varphi_0(x)= \frac{1}{ (1-\beta)x}\big[1-(1-x)^{1-\beta}\big].
		$$
		Then by using Theorem \ref{8theorem1.1}, we conclude the proof.
	\end{proof}
	The value $\gamma=0$ in Theorem \ref{8theorem3.1} gives a result of Kumar and Sahoo \cite{KS2021}. The following limiting case $\beta\to 1$ of Theorem \ref{8theorem3.1} is recently studied in \cite{AG21}.
	\begin{theorem}\label{8theorem3.2}
		If $f(z)=\sum_{n=0}^{\infty}a_n z^n\in \mathcal{B}(\Omega_\gamma)$, then
		$$
		\sum_{n=0}^{\infty}  \bigg(\frac{1}{n+1} \sum_{k=0}^{n} 
		|a_k| \bigg)r^n\leq \frac{1}{r} \log \frac{1}{1-r}
		$$
		for $r\leq R(\gamma)$. Here the number $R(\gamma)$ is the smallest positive root of the equation
		$$
		2x- (3+\gamma)(1-x) \log \frac{1}{1-x}=0
		$$
		that cannot be improved.
	\end{theorem}
	The case $\gamma=0$ in Theorem \ref{8theorem3.2} provides a result of \cite{Kayponn19}.
	\subsection{The $\alpha$-Ces\'{a}ro operator}
	For $\alpha \in \mathbb{C}$ with ${\rm Re}\ \alpha>-1$, the $\alpha$-Ces\'{a}ro operator (see \cite{S94}), considered on the class of analytic functions in $\mathbb{D}$, is defined as
	$$
	\mathcal{C}^\alpha f(z)= \sum_{n=0}^{\infty}\bigg(\frac{1}{A_{n}^{\alpha+1}} \sum_{k=0}^{n} A_{n-k}^{\alpha} a_k \bigg)z^n=(\alpha+1)\int_{0}^{1}f(tz) \frac{(1-t)^\alpha}{(1-tz)^{\alpha+1}}dt,
	$$
	where $f(z)=\sum_{k=0}^{\infty}a_k z^k$ and 
	$
	A_k^\alpha=(\alpha+1)_n/(1)_n.
	$ 
	It follows that 
	$$
	\sum_{k=0}^{\infty} A_{k}^{\alpha} z^k=\frac{1}{(1-z)^{1+\alpha}}.
	$$ 
	Also, after comparing the coefficient of $z^n$ on both sides of the identity
	$$
	\frac{1}{(1-z)^{1+\alpha}}. \frac{1}{(1-z)}=\frac{1}{(1-z)^{2+\alpha}}
	$$
	we obtain
	$$
	A_n^{1+\alpha}=\sum_{k=0}^{n} A_k^{\alpha}, \text{ i.e., } \frac{1}{A_n^{\alpha+1}}\sum_{k=0}^{n} A_{n-k}^{\alpha}=1.
	$$
	Note that $T_1=\mathcal{C}^0$.
	
	The following bound on the operator $\mathcal{C}^\alpha$ over the class $\mathcal{B}(\mathbb{D})$ proved in \cite{Kayponn20}.
	\begin{theorem}
		For $f\in \mathcal{B}(\mathbb{D})$ and $\alpha>-1$, we have
		$$
		|\mathcal{C}^\alpha f(z)|\leq \frac{1+\alpha}{r^{\alpha}} \int_{0}^{r} \frac{t^\alpha}{1-t} dt.
		$$
	\end{theorem}
	Now, we are ready to give another important consequence of Theorem \ref{8theorem1.1}.
	\begin{theorem}\label{7theorem3.4}
		Let $f(z)=\sum_{k=0}^{\infty}a_k z^k$ belongs to $\mathcal{B}(\Omega_\gamma)$ and $\alpha>-1$. Then
		$$
		\sum_{n=0}^{\infty}\bigg(\frac{1}{A_{n}^{\alpha+1}} \sum_{k=0}^{n} A_{n-k}^{\alpha} |a_k| \bigg)r^n\leq \frac{1+\alpha}{r^{\alpha}}\int_{0}^{r} \frac{t^\alpha}{1-t} dt= (\alpha+1)\sum_{n=0}^{\infty} \frac{r^n}{n+\alpha+1},\ \text{ for all } \ r\leq R_\gamma(\alpha), 
		$$
		where $R_\gamma(\alpha)$ is the minimal positive root of the equation
		$$
		(3+\gamma)(1+\alpha)\sum_{n=0}^{\infty} \frac{x^n}{n+\alpha+1}=\frac{2}{1-x}.
		$$
		The number $R_\gamma(\alpha)$ cannot be replaced by a larger constant.
	\end{theorem} 
	\begin{proof} We have
		$$
		\sum_{n=0}^{\infty}\bigg(\frac{1}{A_{n}^{\alpha+1}} \sum_{k=0}^{n} A_{n-k}^{\alpha} |a_k| \bigg)r^n= \sum_{n=0}^{\infty}|a_k|\bigg(\sum_{k=n}^{\infty} \frac{A_{k-n}^{\alpha}}{A_{k}^{\alpha+1}} r^k \bigg).	
		$$
		By considering
		$$
		\varphi_n(r)= \sum_{k=n}^{\infty} \frac{A_{k-n}^{\alpha}}{A_{k}^{\alpha+1}} r^k.
		$$
		It has been obtained in \cite{Kayponn20} that
		$$
		\sum_{k=0}^{\infty} \varphi_k (r)= \frac{1}{1-r} \text{ and } \varphi_0 (r)= (1+\alpha)\sum_{k=0}^{\infty}\frac{r^k}{k+\alpha+1}, \,\ r\in[0,1).
		$$
		Thus our proof is concluded by Theorem \ref{8theorem1.1}. 
	\end{proof}
	
	It is important to mention here that for $\gamma=0$, Theorem \ref{7theorem3.4} reduces to \cite[Theorem 4]{Kayponn20}. Also, the value $\alpha=0$ provides a result of \cite{AG21}. Moreover, the choices $\gamma=0$ and $\alpha=0$ give a result of \cite{Kayponn19}. 
	\subsection{Remark} Let $m$ be a positive integer. The Bernardi operator 
	\cite[P. 11]{MM-Book} is defined by
	$$
	L_\delta [f](z):= \sum_{n=m}^{\infty} \frac{a_n}{n+\delta} z^n
	= \int_{0}^{1} f(zt) t^{\delta-1} dt,
	$$
	for $f(z)=\sum_{n=m}^{\infty} a_n z^n$ with $\delta > -m$.
	
	It is easy to calculate the following sharp bound
	$$
	|L_\delta [f](z)|\leq \frac{1}{m+\delta}r^m, \ |z|=r<1
	$$
	for $f(z)=\sum_{n= m}^{\infty} a_n z^n$.
	
	In Theorem \ref{8theorem1.1} the choice
	$$
	\varphi_n(r)= \frac{r^{n+m}}{n+m+\delta}
	$$
	gives a result of Allu and Ghosh (see \cite[Theorem 2.2]{AG21}). Further, it has been observed in \cite{AG21} that the value $\gamma=0$ leads to a known result of \cite{KS2021}.
	
	

	\bigskip
	\noindent
	{\bf Acknowledgement.} 
	I would like to thank my Ph.D. supervisor Dr. Swadesh Kumar Sahoo for his helpful remarks.
	The work of the author is supported by CSIR, New Delhi (Grant No: 09/1022(0034)/2017-EMR-I).

	\medskip
	\noindent
	{\bf Conflict of Interests.} The authors declare that there is no conflict of interests regarding the publication of this paper.

\end{document}